\documentclass[
BCOR2pt, 
captions=nooneline, 
bibliography=totoc, 
numbers=noenddot, 
parskip=half, 
headings=normal, 
abstracton 
]{scrartcl} 

\usepackage[USenglish]{babel} 
\usepackage{graphicx} 
\usepackage{tikz}
\usepackage{framed}
\usepackage{faktor}
\usepackage{floatrow}
\usepackage{remreset} 
\usepackage[nouppercase]{scrpage2} 
\usepackage{amsmath} 
\usepackage{amssymb} 
\usepackage[thmmarks, 
amsmath, 
hyperref 
]{ntheorem} 
\usepackage{bm} 
\usepackage{bbm} 
\usepackage{enumitem} 
\usepackage[hang]{subfigure} 
\usepackage{wrapfig} 
\usepackage{tabularx} 
\usepackage{color}
\usepackage{dcolumn} 
\usepackage{booktabs} 
\usepackage{listings} 
\usepackage{psfrag} 
\usepackage[pdftex, 
setpagesize=false, 
pdfborder={0 0 0}, 
pdfpagemode=UseOutlines, 
]{hyperref} 

\makeatletter
\newcommand\mytoday{\number\year-\ifcase\month\or 01\or 02\or 03\or 04\or 05\or 06\or 07\or 08\or 09\or 10\or 11\or 12\fi-\ifcase\day\or 01\or 02\or 03\or 04\or 05\or 06\or 07\or 08\or 09\or 10\or 11\or 12\or 13\or 14\or 15\or 16\or 17\or 18\or 19\or 20\or 21\or 22\or 23\or 24\or 25\or 26\or 27\or 28\or 29\or 30\or 31\fi} 
\makeatother
\setkomafont{sectioning}{\normalcolor\bfseries} 
\clearscrheadfoot 
\automark[section]{subsection} 
\setkomafont{captionlabel}{\bfseries} 
\newcolumntype{d}[2]{D{.}{.}{#1.#2}} 
\newcommand*{\abstractnoindent}{} 
\let\abstractnoindent\abstract
\renewcommand*{\abstract}{\let\quotation\quote\let\endquotation\endquote
\abstractnoindent}
\makeatletter
\renewcommand{\p@enumii}[1]{\theenumi(#1)}
\makeatother


\theoremstyle{break} 
\theoremheaderfont{\bfseries}
\theorembodyfont{}
\theoremseparator{}
\newtheorem{definition}{Definition}[section] 
\newtheorem{lemma}[definition]{Lemma}

\newtheorem{remark}[definition]{Remark}
\newtheorem{example}[definition]{Example}

\theoremstyle{nonumberbreak} 
\theoremsymbol{$\Box$}
\newtheorem{proof}{Proof}

\newcommand*{\IP}{\mathbb{P}}
\newcommand*{\I}{\mathbb{I}}
\newcommand*{\IE}{\mathbb{E}}
\newcommand*{\IR}{\mathbb{R}}

\newcommand*{\IN}{\mathbb{N}}

\newcommand*{\IQ}{\mathbb{Q}}
\newcommand*{\F}{\mathcal{F}}

\usetikzlibrary{snakes}
\usetikzlibrary{arrows,shapes,positioning}
\usetikzlibrary{decorations.markings}
\tikzstyle arrowstyle=[scale=2]
\tikzstyle directed=[postaction={decorate,decoration={markings,
    mark=at position 1 with {\arrow[arrowstyle]{stealth}}}}]

\begin{document}

\renewcommand{\figurename}{Fig.}

\thispagestyle{plain}
	\begin{center}
		{\bfseries\Large On the structure of exchangeable extreme-value copulas}
		\par
		\vspace{1cm}
		{\Large Jan-Frederik Mai and Matthias Scherer \footnote{Technical University of Munich, Department of Mathematics, Chair of Mathematical Finance, Parkring 11, 85748 Garching--Hochbr\"uck, email: \texttt{mai@tum.de} and \texttt{scherer@tum.de}}}
		\vspace{0.2cm}
		\\
	\end{center}
	\begin{center}
Version of \today.
\end{center}
\begin{abstract}
We show that the set of $d$-variate symmetric stable tail dependence functions, uniquely associated with exchangeable $d$-dimensional extreme-value copulas, is a simplex and determine its extremal boundary. The subset of elements which arises as $d$-margins of the set of $(d+k)$-variate symmetric stable tail dependence functions is shown to be proper for arbitrary $k \geq 1$. Finally, we derive an intuitive and useful necessary condition for a bivariate extreme-value copula to arise as bi-margin of an exchangeable extreme-value copula of arbitrarily large dimension, and thus to be conditionally iid.
\end{abstract}

\section{Introduction}
The problem of determining whether a given exchangeable probability law on $\IR^d$ can arise as a $d$-dimensional margin of some exchangeable probability law on $\IR^{d+k}$, $k \geq 1$, is known as the extendibility problem in the literature. If there is a solution to the extendibility problem for arbitrary $k \geq 1$, one calls the probability law \emph{(infinitely) extendible}. In the general case, that is without postulating any additional conditions on the involved probability distributions, \cite{konstantopoulos19} derive an analytical criterion to check for extendibility, although this criterion is difficult to apply in concrete cases. Analytical solutions of the infinite extendibility problem have natural connections with Harmonic Analysis, rendering the topic interesting for theorists. But the problem is also interesting for applied probabilists, since infinitely extendible models can be used as flexible dependence models that are still very convenient to work with in large dimensions. The most famous solutions to the infinite extendibility problem for specific families of distributions comprise $\ell_2$-norm symmetric laws (Schoenberg's Theorem), $\ell_1$-norm symmetric laws associated with Archimedean copulas (see \cite{neilnev09}), and $\ell_{\infty}$-norm symmetric laws (see \cite{gnedin95}), see also \cite{rachev91} for a nice wrapping of these three popular cases and a generalization to $\ell_{p}$-norm symmetric laws for arbitrary $p \in (0,\infty]$. More recently, the infinite extendibility problem has also been solved for popular families of multivariate exponential and geometric distributions, see \cite{mai13,maischerer13brazil}, and has also been dealt with for extreme-value distributions in \cite{maischerer13,mai18b}. Recall further that the seminal De Finetti Theorem implies that the notions ``infinitely extendible'' and ``conditionally iid'' coincide for exchangeable probability laws, see \cite{definetti37}.
\par
The present article may be seen as a continuation of the work in \cite{maischerer13,mai18b} and deals with the extendibility problem for exchangeable extreme-value copulas. More specifically, we investigate which exchangeable $d$-variate extreme-value copulas arise as $d$-margins of some $(d+k)$-variate exchangeable extreme-value copula, and which do not. Recall that an extreme-value copula $C:[0,1]^d \rightarrow [0,1]$ is (the restriction to $[0,1]^d$ of) a distribution function with one-dimensional margins that are uniform on $[0,1]$, and which satisfies the extreme-value property $C(u^t_1,\ldots,u^t_d)=C(u_1,\ldots,u_d)^t$ for all $u_1,\ldots,u_d\in[0,1]$ and $t\geq 0$. The extreme-value property analytically characterizes multivariate distribution functions that can arise as limits of appropriately normalized componentwise maxima/minima of independent and identically distributed random vectors, see \cite{resnick87} for a textbook account on multivariate extreme-value theory. The restriction to uniform one-dimensional margins, i.e.\ to copulas, instead of arbitrary extreme-value distribution functions is without loss of generality, since by virtue of Sklar's Theorem we can write an arbitrary distribution function of a $d$-variate extreme-value distribution as $F(x_1,\ldots,x_d):= C\big( F_1(x_1),\ldots,\,F(x_d)\big)$, where $C:[0,1]^d \rightarrow [0,1]$ is an extreme-value copula and $F_1,\ldots,F_d$ denote the one-dimensional margins, which are necessarily one-dimensional extreme-value distribution functions. We refer to \cite{durante2015principles, joe2014dependence} for general background on copulas, and to \cite{joe97} for a book on copulas with a specific emphasis on extreme-value copulas. 
\par
Regarding the important special case of infinite extendibility, we say that a $d$-variate exchangeable extreme-value copula $C$ is infinitely extendible if it arises as $d$-margin of some $(d+k)$-dimensional exchangeable extreme-value copula for arbitrary $k \geq 1$. \cite{maischerer13} show that this is the case if and only if there is a right-continuous, non-decreasing stochastic process $H=\{H_t\}_{t \geq 0}$ satisfying $H_0=0$, $\lim_{t \rightarrow \infty}H_t=\infty$, and $-\log(\IE[\exp(-H_1)])=1$, which is strongly infinitely divisible with respect to time, such that $C(\vec{u})=\IP(U_1 \leq u_1,\ldots,U_d \leq u_d)$ for $\vec{u}=(u_1,\ldots,u_d) \in [0,1]^d$, where
\begin{gather}
U_k:=\exp(-X_k),\quad X_k:=\inf\{t>0\,:\,H_t> \epsilon_k\},\quad k=1,\ldots,d,
\label{condiidmodel}
\end{gather}
and $\epsilon_1,\ldots,\epsilon_d$ are independent and identically distributed with standard exponential distribution, independent of $H$. \emph{Infinite divisibility with respect to time} means that for arbitrary $n \in \IN$ the stochastic process $H$ is identical in law with the stochastic process $H^{(1)}_{./n}+\ldots+H^{(n)}_{./n}$ for independent copies $H^{(1)},\ldots,H^{(n)}$ of $H$. Making use of this result, \cite{mai18b} further shows that $C$ is infinitely extendible if and only if there is a pair $(b,\lambda)$ comprising a constant $b \in[0,1]$ and a probability measure $\lambda$ on the set $\F_1$ of distribution functions of non-negative random variables with unit mean such that
\begin{gather}
 C(\vec{u})= \Big(\prod_{k=1}^{d}u_k\Big)^b\,\exp\Big\{- (1-b)\,\int_{\F_1}\,\int_0^{\infty}1-\prod_{k=1}^{d}F\Big( \frac{s}{-\log(u_k)}\Big)\,\mathrm{d}s\,\lambda(\mathrm{d}F) \Big\}.
\label{condiidcopula}
\end{gather}
The pair $(b,\lambda)$ can be used to specify the process $H$ in (\ref{condiidmodel}) as
\begin{gather*}
H_t = b\,t+(1-b)\,\sum_{k \geq 1}-\log\Big\{ G_k\Big( \frac{\eta_1+\ldots+\eta_k}{t}-\Big)\Big\},\quad t \geq 0,
\end{gather*}
where $G_1,G_2,\ldots$ is a sequence of independent and identically distributed random distribution functions drawn from $\gamma$ and $\eta_1,\eta_2,\ldots$ is an independent sequence of independent and identically distributed standard exponential variates. The stochastic model (\ref{condiidmodel}) for infinitely extendible extreme-value copulas is very convenient for applications due to the fact that $U_1,\ldots,U_d$ are independent and identically distributed conditioned on the $\sigma$-algebra generated by $H$. Unfortunately, it is in general unclear whether a given exchangeable extreme-value copula is of the form (\ref{condiidcopula}) and thus admits the convenient stochastic representation (\ref{condiidmodel}).
\par
With this background in mind, the present article finally provides the answer to a natural question about which the authors were pondering for almost a decade. It has served as a fruitful source of inspiration, since we have been able to discover plenty of related results in the meantime while chasing our `Moby Dick', which is: 
\begin{quote}
\textit{Is an exchangeable extreme-value copula always of the form (\ref{condiidcopula}), and thus admits a stochastic representation as in (\ref{condiidmodel})?}
\end{quote}
We are going to answer this question with `no' in this article, for arbitrary dimension $d \geq 2$. Whereas for $d=3$ (and from this easily also for $d \geq 3$) already the remark after Example 1 of \cite{mai18b} shows that not every exchangeable extreme-value copula is of the form (\ref{condiidcopula}), for $d=2$ this question has been open until now. Our strategy of proof provides a structural result of independent interest for exchangeable extreme-value copulas in arbitrary dimension $d \geq 2$: the set of symmetric stable tail dependence functions is a simplex. We determine its extremal boundary and show that extremal elements cannot arise as margins of a higher-dimensional symmetric stable tail dependence function.

\subsection{How could Moby Dick escape us for so long?}
It is well known that infinitely extendible random vectors are necessarily positively associated in some sense. In fact, \cite{shaked77} even calls such random vectors \emph{positive dependent by mixture (PDM)}, which is a third synonym that is found in the literature for ``conditionally iid'' or ``infinitely extendible''. Thus, a popular strategy to prove that a random vector is not infinitely extendible is to show that its components exhibit some sort of negative association, which is successful for many popular families of distributions for which exchangeable but not infinitely extendible examples are known. What makes the investigation of extreme-value copulas more delicate in this regard is the fact that extreme-value copulas always induce non-negative association, see \cite[Theorem 6.7, p.\ 177]{joe97}, which disqualifies the aforementioned strategy of proof. Furthermore, well-known parametric families of exchangeable bivariate extreme-value copulas, like the Gumbel copula, the Galambos copula, or the Cuadras--Aug\'e copula\footnote{Which equals the exchangeable, bivariate Marshall--Olkin copula, see also Example \ref{ex_CA}.}, are indeed of the form (\ref{condiidcopula}).
\subsection{Why is our Moby Dick relevant?}
The property of being infinitely extendible has far-reaching consequences in applications. If a model is not infinitely extendible, it is difficult, or even impossible, to add more components to the existing model while preserving its structure. There are many practical applications where it is natural and required to let the dimension $d$ vary. Consider, for example, an insurance company (resp.\ a bank) that has $d$ insured objects under contract (resp.\ has issued $d$ loans). Here, the quantity $d$ changes frequently in a natural way and it would be inconvenient or impracticable to use a model that is limited to an upper bound for $d$. Moreover, as already mentioned, a conditionally iid structure allows to apply classical limit theorems like the law of large numbers or Glivenko--Cantelli in a conditional version, which can be a valuable tool in applications. 
\par
From a theoretical perspective the extendibility problem is interesting and challenging, as the canonical stochastic model of the random vector in concern might not at all be related to the concept of conditional independence. The stochastic model (\ref{condiidmodel}) is by definition based on a two-step Bayesian simulation: first simulate the latent factor $H$, second simulate an iid sequence drawn from the distribution function $1-\exp(-H)$. In contrast, the canonical stochastic model for arbitrary (not necessarily exchangeable or even infinitely extendible) extreme-value copulas, based on limits of component-wise minima/maxima, does a priori not have any connection to conditional independence.

\subsection{Organization of the article and notations}
Section~\ref{Section_2} is concerned with the algebraic structure of exchangeable $d$-variate extreme-value copulas. The main contribution of Sections~\ref{Section_2.1} and \ref{Section_2.2} is the observation that the associated family of symmetric stable tail dependence functions forms a simplex. In Section~\ref{Sec_impossible} it is shown that certain extremal stable tail dependence functions cannot arise as lower-dimensional margins of higher-dimensional symmetric stable tail dependence functions, thus proving that the notion ``infinitely extendible'' and ``exchangeable'' are not synonyms in the realm of extreme-value copulas. Section~\ref{Section_3} provides a non-trivial and intuitive necessary condition for a bivariate extreme-value copula to be infinitely extendible.
\par
Throughout, we denote equality in distribution by $\stackrel{d}{=}$, the symbol $\sim$ means ``is distributed according to,'' and the acronym \emph{iid} means \underline{i}ndependent and \underline{i}dentically \underline{d}istributed. We say that a probability measure $\mu$ on $\IR^d$ is \emph{conditionally iid} if there is a random vector $\bm{X}=({X}_1,\ldots,{X}_d) \sim \mu$ on some probability space $(\Omega,{\F},{\IP})$ such that ${X}_1,\ldots,{X}_d$ are iid conditioned on some $\sigma$-algebra $\mathcal{T} \subset {\F}$. In this case, we also say that $\bm{X}$ itself and also its distribution function are \emph{conditionally iid}. Conditioned on $\mathcal{T}$, and thus also unconditioned, the limiting process $d \rightarrow \infty$ is clearly viable by the natural product space extension, and thus we may view $({X}_1,\ldots,{X}_d)$ as the first $d$ members of an infinite sequence $\{{X}_k\}_{k \in \IN}$ whose components are iid conditioned on ${\mathcal{T}}$. For this reason, the notion ``conditionally iid'' is sometimes also referred to as ``(infinitely) extendible'' in the literature on exchangeable probability laws, for instance in \cite{spizzichino82}. 

\section{Exchangeable extreme-value models}\label{Section_2}
We analyze the anatomy of exchangeable extreme-value copulas. We shall find that the family of symmetric stable tail dependence functions, associated with the exchangeable subfamily of extreme-value copulas, forms a simplex. This is an interesting finding, as the full class of (not necessarily symmetric) stable tail dependence functions does not share this property. Regarding the organization, in Section \ref{Section_2.1} we first present the simpler bivariate case separately, as the terminologies in the literature for the bivariate and multivariate treatment differ, and also because we pick up some of the specific two-dimensional notations in Section \ref{Section_3} below.

\subsection{Bivariate extreme-value copulas: Characterization and geometry}\label{Section_2.1}
	A convenient analytical description of a bivariate extreme-value copula is available in terms of its so-called \emph{(Pickands) dependence function} $A:[0,1] \rightarrow [1/2,1]$, which satisfies $1\geq A(x) \geq \max\{1-x,x\}$, $A(0)=A(1)=1$, and is convex. We denote the set of all such functions by $\mathfrak{A}$. To wit, $A \in \mathfrak{A}$ if and only if the function 
\begin{align*}
C(u_1,u_2) &= (u_1\,u_2)^{A\big( \frac{\log(u_1)}{\log(u_1\,u_2)}\big)},\quad u_1,u_2 \in [0,1],
\end{align*}
is a bivariate extreme-value copula. Furthermore, $C$ is exchangeable if and only if $A(x)=A(1-x)$ for all $x \in [0,1]$, and we denote the respective subset of $\mathfrak{A}$ by $\mathfrak{A}^X$ in the sequel. Clearly, $2$-margins of some exchangeable $d$-variate extreme-value copula for $d\geq 3$ are exchangeable bivariate extreme-value copulas. The subset of those $A$ corresponding to bivariate extreme-value copulas which arise as $2$-margins of an exchangeable $d$-variate extreme-value copula for arbitrarily large $d \geq 2$ will be denoted $\mathfrak{A}^{\ast}$. Notice that elements in $\mathfrak{A}^{\ast}$ are in particular conditionally iid, and in this article we establish that the inclusion $\mathfrak{A}^{\ast} \subsetneq \mathfrak{A}^X$ is proper. 
\par
Elements in $\mathfrak{A}$ can furthermore be associated uniquely with random variables $Q$ on $[0,1]$ with $\IE[Q]=1/2$. To wit, for such $Q$ the function
\begin{align}
A(x) = 2\,\IE\big[\max\{x\,Q,(1-x)\,(1-Q)\}\big], \quad x \in [0,1],
\label{Pick_biv}
\end{align}
lies in $\mathfrak{A}$, every $A \in \mathfrak{A}$ can be represented by such $Q$, and the law of $Q$ is unique, see \cite{dehaan84,ressel13}. Clearly, $A \in \mathfrak{A}^X$ if and only if the associated random variable $Q$ satisfies $Q \stackrel{d}{=}1-Q$. For historical reasons, the finite measure $2\,\IP\big( (Q,1-Q) \in \mathrm{d}\vec{q}\big)$ on the two-dimensional unit simplex $S_2:=\{\vec{q}=(q_1,q_2) \in [0,1]^2\,:\,q_1+q_2=1\}$ is called \emph{Pickands dependence measure} of $C$, see \cite{Pickands1981}. It is educational to recall that
\begin{align*}
A(x) = 1+\int_0^{x}1-2\,\IP(Q \leq 1-y)\,\mathrm{d}y,
\end{align*}
which follows from straightforward computations, using integration by parts and $\IE[Q]=1/2$. This shows in particular that $A$ is convex with derivative from the right being given by $x \mapsto 1-2\,\IP(Q \leq 1-x)$, so that studying properties of $A$ is tantamount to studying properties of $Q$ (which is tantamount to studying properties of $C$).

\begin{example}[The family $BC_2$] \label{ex_BC2}
For $0 \leq a \leq 1/2 \leq b \leq 1$ we consider a random variable $Q=Q_{a,b}$ with $\IP(Q_{a,b}=a)=(b-1/2)/(b-a)=1-\IP(Q_{a,b}=b)$, with the convention $0/0=1$ in case $a=b=1/2$, and observe that the associated Pickands dependence function is given by
\begin{align*}
A_{a,b}(x) = \max\{a\,x,b\,(1-x)\}+\max\{(1-a)\,x,(1-b)\,(1-x)\}.
\end{align*}
This forms the extreme-value copula family $BC_2$, introduced in \cite{maischerer11}.
\end{example}

\begin{remark}[$\mathfrak{A}$ is not a simplex]
It has been shown in \cite{maischerer11}, and later also in \cite{trutschnig16} with an alternative proof, that $\mathfrak{A}$ equals the convex hull of its extremal boundary $\partial_e \mathfrak{A}=\{A_{a,b} \,:\,0 \leq a \leq 1/2 \leq b \leq 1\}$. Combined with the fact that the law $Q_{a,b}$ associated with $A_{a,b}$ via (\ref{Pick_biv}) has at most two atoms, the nomenclature $BC_2$ is the abbreviation of \underline{b}uilding \underline{c}omponents with at most \underline{$2$} atoms. However, \cite{beran14} provide an example showing that $\mathfrak{A}$ is not a simplex, i.e.\ a representation of $A \in \mathfrak{A}$ as convex combination of elements in $\partial_e \mathfrak{A}$ is not necessarily unique. They point out that
\begin{gather*}
A(x)=\frac{1}{4}\,A_{0,1}(x)+\frac{3}{4}\,A_{\frac{1}{3},\frac{2}{3}}(x) = \frac{1}{2}\,A_{0,\frac{2}{3}}(x)+\frac{1}{2}\,A_{\frac{1}{3},1}(x),
\end{gather*}
as can readily be checked. The situation changes when restricting one's attention to the symmetric subfamily, as we will see.
\end{remark}

Obviously, $A_{a,b}$ in Example \ref{ex_BC2} is symmetric if and only if $b=1-a$, and for $a \in (1/2,1]$ we introduce the auxiliary notation $A_{a,1-a} := A_{1-a,a}$ for the sake of convenience in the upcoming computation. If $A \in \mathfrak{A}^X$ is symmetric, the associated random variable $Q $ satisfies $Q \stackrel{d}{=}1-Q$, which implies that
\begin{align*}
A(x) &= 2\,\IE[\max\{x\,Q,(1-x)\,(1-Q)\}]  \\
& = \IE[\max\{x\,Q,(1-x)\,(1-Q)\}+\max\{x\,(1-Q),(1-x)\,Q\}] \\
&= \IE[A_{Q,1-Q}(x)]= \int_{\big[ 0,\frac{1}{2}\big]}A_{q,1-q}(x)\,\nu(\mathrm{d}q),
\end{align*}
where the measure $\nu$ is defined and uniquely determined by the law of $Q$ via $\nu([0,q)):=2\,\IP(Q < q)$ for $q \in [0,1/2)$, or more precisely by
\begin{gather*}
\nu\Big(\Big\{\frac{1}{2}\Big\}\Big)=\IP\Big(Q=\frac{1}{2}\Big),\quad \nu(B):=2\,\IP(Q \in B),\quad \forall \,B \subset [0,1/2) \mbox{ Borel sets}.
\end{gather*}
This shows that the set $\mathfrak{A}^X$ of symmetric dependence functions equals the convex hull of $\partial_e \mathfrak{A}^X=\{A_{a,1-a}\,:\,0 \leq a \leq 1/2\}$. Furthermore, since the representing measure $\nu$ is unique (since the law of $Q$ is known to be unique), $\mathfrak{A}^X$ is a simplex, indicating that the algebraic structure in the exchangeable case is a lot nicer than in the non-exchangeable case. We label this remarkable observation a lemma.

\begin{lemma}[Bivariate case: $\mathfrak{A}^X$ is a simplex]\label{lemma_simplex1}
$\mathfrak{A}^X$ is a simplex with extremal boundary $\partial_e \mathfrak{A}^X=\{A_{a,1-a}\,:\,0 \leq a \leq 1/2\}$.
\end{lemma}

\subsection{Multivariate extreme-value copulas: Characterization and geometry}\label{Section_2.2}

We generalize Lemma \ref{lemma_simplex1} to the multivariate case. For an extreme-value copula $C$, the function 
\begin{gather*}
\ell(\vec{x}):=-\log\Big( C\big( e^{-x_1},\ldots,e^{-x_d}\big) \Big),\quad \vec{x}=(x_1,\ldots,x_d) \in [0,\infty)^d,
\end{gather*}
is called \emph{stable tail dependence function}. The extreme-value property of $C$ implies that $\ell$ is homogeneous of order one, i.e.\ $\ell(t\,\vec{x})=t\,\ell(\vec{x})$. Clearly, the function $\ell$ uniquely determines its associated extreme-value copula $C$. In fact, due to the homogeneity already the restriction of $\ell$ to the $d$-dimensional unit simplex $S_d:=\{\vec{x} \in [0,1]^d\,:\,x_1+\ldots+x_d=1\}$ determines $C$. As a generalization of the bivariate case the classical result of \cite{Pickands1981} in the multivariate case states that $\ell$ is the stable tail dependence function of some $d$-variate extreme-value copula if and only if there exists a random vector $\vec{Q}=(Q_1,\ldots,Q_d)$, taking values in $S_d$ and satisfying $\IE[Q_k]=1/d$ for each component $k$, such that 
\begin{gather}
\ell(\vec{x})=d\,\IE[\max\{x_1\,Q_1,\ldots,x_d\,Q_d\}]. 
\label{Pickkmult}
\end{gather}
This random vector $\vec{Q}$ is further uniquely determined in distribution, see \cite{dehaan84,ressel13}. In the bivariate case $d=2$ we recall that the (Pickands) dependence function is obtained from the stable tail dependence function as $A(x)=\ell(x,1-x)$. A $d$-dimensional extreme-value copula is exchangeable if and only if $\ell$ is symmetric, meaning that $\ell$ is invariant with respect to permutations of its arguments.
\par
The (convex) set of all stable tail dependence functions is not a simplex, as already highlighted in the bivariate case. But like in the bivariate case we are going to show that the (convex) subset of all symmetric stable tail dependence functions, denoted by $\mathfrak{L}^{X}_d$ henceforth, is a simplex. To this end, we denote by $\mathfrak{P}_d$ the set of permutations on $\{1,\ldots,d\}$ and introduce an equivalence relation on $S_d$: 
\begin{gather*}
\vec{q} \sim \vec{r} \mbox{, if there exists a }\sigma \in \mathfrak{P}_d \mbox{ such that }\vec{q}=(r_{\sigma(1)},\ldots,r_{\sigma(d)}). 
\end{gather*}
The equivalence class of $\vec{q}$ is denoted by $[\vec{q}]$, and we denote the set of equivalence classes by $\faktor{S_d}{\sim}$. Since $S_d$ is compact Hausdorff, and since each equivalence class is a set of at most $d!$ elements, the quotient space $\faktor{S_d}{\sim}$ is also compact Hausdorff. Thus, Radon probability measures on $\faktor{S_d}{\sim}$ are well-defined. In fact, if $q:S_d \rightarrow \faktor{S_d}{\sim}$ denotes the quotient map, then we may define a probability measure on $\faktor{S_d}{\sim}$ from a probability measure $\mu$ on $S_d$ by $[\mu](B):=\mu( q^{-1}(B))$, where $q^{-1}(B)$ is the pre-image of an open set $B \subset \faktor{S_d}{\sim}$. Intuitively, if $\vec{Q}\sim \mu$ then $[\vec{Q}] \sim [\mu]$. Conversely, every probability measure on $\faktor{S_d}{\sim}$ is obtained in such way.
\begin{lemma}[Multivariate case: $\mathfrak{L}^{X}_d$ is a simplex]\label{lemma_simplexgen}
For each $\vec{q} \in S_d$ the function
\begin{gather*}
\ell^{(d)}_{\vec{q}}(\vec{x}) := \frac{1}{(d-1)!}\,\sum_{\sigma \in \mathfrak{P}_d}\max\{x_1\,q_{\sigma(1)},\ldots,x_d\,q_{\sigma(d)}\}
\end{gather*}
is in $\mathfrak{L}^{X}_d$ and depends on $\vec{q}$ only through its equivalence class $[\vec{q}]$, thus we may write $\ell^{(d)}_{\vec{q}}=\ell^{(d)}_{[\vec{q}]}$. Further, $\mathfrak{L}^{X}_d$ is a simplex with extreme points $\{\ell^{(d)}_{[\vec{q}]}\}_{[\vec{q}] \in \faktor{S_d}{\sim}}$. 
\end{lemma}
\begin{proof}
First, it is easy to see that $\ell^{(d)}_{\vec{q}}$ is the stable tail dependence function associated with the random vector $\vec{Q}=(q_{\Sigma(1)},\ldots,q_{\Sigma(d)})$, where $\Sigma$ is a random variable uniformly distributed on $\mathfrak{P}_d$, i.e.\ with $\IP(\Sigma=\sigma)=1/d!$ for each $\sigma \in \mathfrak{P}_d$. Notice that $\IE[Q_k]=1/d$ for each $k=1,\ldots,d$. Since $\vec{Q}$ is obviously exchangeable, $\ell_{\vec{q}}^{(d)}$ is symmetric. By definition, $\vec{q}$ depends on $\vec{q}$ only through $[\vec{q}]$.
\par 
Next, we observe that $\ell \in \mathfrak{L}^{X}_d$ if and only if the random vector $\vec{Q}$ associated with its unique Pickands measure is exchangeable. While sufficiency is obvious, to see necessity we let $\sigma \in \mathfrak{P}_d$ be arbitrary and consider $\vec{R}:=(Q_{\sigma(1)},\ldots,Q_{\sigma(d)})$. We have to show that $\vec{R} \stackrel{d}{=} \vec{Q}$. To this end, we observe that $\vec{R}$ is also a Pickands measure of some stable tail dependence function $\ell_R$. Using the symmetry of $\ell$, however, we see that
\begin{align*}
\ell(x_{\sigma(1)},\ldots,x_{\sigma(d)})&=\ell(\vec{x})=d\,\IE[\max_{i=1,\ldots,d}\{x_{i}\,Q_i\}]\\
&=d\,\IE[\max_{i=1,\ldots,d}\{x_{\sigma(i)}\,Q_{\sigma(i)}\}] = \ell_R(x_{\sigma(1)},\ldots,x_{\sigma(d)}).
\end{align*}
Thus, $\ell_R=\ell$, which implies $\vec{R} \stackrel{d}{=} \vec{Q}$, as the law of $\vec{Q}$ is uniquely determined by $\ell$.
\par
Finally, let $\ell \in \mathfrak{L}^{X}_d$ be arbitrary. We have just seen that its associated random vector $\vec{Q}$ is exchangeable. Its probability law, denoted $\mu$, is a law on $S_d$. We consider the associated probability law $[\mu]$ on $\faktor{S_d}{\sim}$, which is the law of $[\vec{Q}]$, and see, using exchangeability of $\vec{Q}$ in $(\ast)$ below,
\begin{align*}
\ell(\vec{x}) &=\frac{1}{(d-1)!}\,d!\,\IE[\max_{i=1,\ldots,d}\{x_i\,Q_i\}] \stackrel{(\ast)}{=} \frac{1}{(d-1)!}\,\sum_{\sigma \in \mathfrak{P}_d}\IE[\max_{i=1,\ldots,d}\{x_i\,Q_{\sigma(i)}\} ]\\
& = \IE\big[ \ell_{\vec{Q}}^{(d)}(\vec{x})\big] = \IE\big[ \ell_{[\vec{Q}]}^{(d)}(\vec{x})\big] = \int_{\faktor{S_d}{\sim}}\ell_{[\vec{q}]}^{(d)}(\vec{x})\,[\mu](\mathrm{d}[\vec{q}]).
\end{align*}
Since $\vec{Q}$ is exchangeable and the law of $\vec{Q}$ on $S_d$ is unique, the law of $[\vec{Q}]$ on $\faktor{S_d}{\sim}$ is unique as well, finishing the argument.
\end{proof}


\subsection{Impossibility of embedding the boundary into a higher dimension}\label{Sec_impossible}

We fix $d\geq n+1\geq 3$ in this paragraph. We show that $\ell^{(n)}_{[\vec{q}]}\in \mathfrak{L}^X_n$ cannot be the $n$-margin of any $\ell\in \mathfrak{L}^X_d$ for $\vec{q}\in S_n$ with $0<q_1<\ldots<q_n<1$. This, in turn, shows that not all $n$-variate exchangeable extreme-value copulas arise as $n$-margins of higher-dimensional exchangeable extreme-value copulas. In particular, not all $n$-variate exchangeable extreme-value copulas are conditionally iid.

\begin{lemma}[Impossibility of embedding the boundary]\label{Theorem_impossible}
Let $0<q_1<\ldots<q_n<1$, $n\geq 2$. We claim that $\ell^{(n)}_{[\vec{q}]}\in \mathfrak{L}^X_n$ cannot be an $n$-margin of any $\ell\in \mathfrak{L}^X_d$ for $d\geq n+1$.
\end{lemma}
\begin{proof}
We compute the $n$-margin of an arbitrary $\ell \in \mathfrak{L}^X_d$ and denote by $\vec{Q}$ the random vector associated with $\ell$ via (\ref{Pickkmult}). In this computation, we abbreviate $S:=\sum_{j=1}^n Q_j$, noticing that $\IP(S>0)>0$ by exchangeability of $\vec{Q}$ and the fact that $Q_1+\ldots+Q_d=1$ almost surely.
\begin{align*}
\ell(x_1,\ldots,x_n,0,\ldots,0)&=d\,\IE\big[\max\{x_1\,Q_1,\ldots, x_n\,Q_n\}\big]\\
&=d\,\IE\big[\I_{\{S>0\}}\,S\,\max\{x_1\,\frac{Q_1}{S},\ldots, x_n\,\frac{Q_n}{S}\} \big]\\
&\stackrel{(a)}{=}d\,\IP(S>0)\,\IE_{\tilde{\IP}}\big[S\,\max\{x_1\,\frac{Q_1}{S},\ldots, x_n\,\frac{Q_n}{S}\} \big]\\
&\stackrel{(b)}{=}n\,\IE_{\IQ}\big[\max\{x_1\,\frac{Q_1}{S},\ldots, x_n\,\frac{Q_n}{S}\} \big].
\end{align*}
Above, in (a) we introduce the new probability measure $\tilde{\IP}(B):=\IP(B|S>0)$ s.t.\ $\IE[\I_{\{S>0\}}\,X]=\IE[X|S>0]\,\IP(S>0)=\IE_{\tilde{\IP}}[X]\,\IP(S>0)$ for arbitrary random variables $X$, and in (b) we introduce the probability measure $\IQ$ defined by $\mathrm{d}\IQ:=\frac{d\,\IP(S>0)\,S}{n}\,\mathrm{d}\tilde{\IP}$. Notice that $\IQ$ has the same null sets as $\tilde{\IP}$, since $\tilde{\IP}(S>0)=1$ and $\IQ$ is indeed a probability measure, since
\begin{gather*}
\IQ(\Omega) = \IE_{\tilde{\IP}}\Big[ \frac{d\,\IP(S>0)\,S}{n}\Big] = \frac{d\,\IP(S>0)}{n}\,\IE_{\tilde{\IP}}[S] = \frac{d}{n}\,\IE[S\,\I_{\{S>0\}}] = \frac{d}{n}\,\IE[S] = 1.
\end{gather*}
Hence, the Pickands measure of an $n$-margin of $\ell$ is given by $n\,\IQ\Big(\big(\frac{Q_1}{S},\ldots,\frac{Q_n}{S}\big)\in\mathrm{d}\vec{x}\Big)$ on $S_n$. We denote by $\mathcal{S}$ the support of this measure, i.e.\ $\mathcal{S}\subset S_n$ is a closed subset and $\vec{s}:=(s_1,\ldots,s_n)\in \mathcal{S}$ if and only if $\IQ\Big(\big(\frac{Q_1}{S},\ldots,\frac{Q_n}{S}\big)\in U\Big)>0$ for an arbitrary neighbourhood $U$ of $\vec{s}$, $\vec{s}\in U\subset S_n$. We know that the support of $\ell_{[\vec{q}]}$ equals $[\vec{q}]\subset S_n$ and we claim that $\mathcal{S}\neq [\vec{q}]$, which implies the main claim. So, we (falsely) assume that $\mathcal{S}=[\vec{q}]$ and show that this implies a contradiction.
\par 
We denote by $\mathcal{R}$ the support on $S_d$ of the Pickands measure of $\ell$, which equals the support of $\vec{Q}$ under $\IP$. Note that by construction, the support of $\vec{Q}$ under $\IP$ and $\IQ$ can at most differ at values of $S_d$ with at least one component equal to zero, which is irrelevant for us in the following. Let $\vec{r}\in\mathcal{R}$ be arbitrary. By exchangeability of $\vec{Q}$, $[\vec{r}]\subset \mathcal{R}$. If some component $r_i$ of $\vec{r}$ is zero, this contradicts $\min_i\{q_i\} \geq q_1>0$, so all components of $\vec{r}$ are positive. If $r_i=r_j$ for some $i\neq j$ then 
\begin{align*}
\IQ\Big(\big(\frac{Q_1}{S},\ldots,\frac{Q_n}{S}\big)\text{ has two identical components }\Big)>0
\end{align*} which contradicts $q_1<\ldots<q_n$ and the assumption $\mathcal{S}=[\vec{q}]$. Thus, we may assume without loss of generality that $0<r_1<\ldots<r_d<1$ and observe
\begin{align*}
\mathcal{S}\supset\Big\{\big(\frac{r_{\sigma(1)}}{\sum_{j=1}^nr_{\sigma(j)}},\ldots, \frac{r_{\sigma(n)}}{\sum_{j=1}^nr_{\sigma(j)}}\big):\sigma\in \mathfrak{P}_d\Big\}=:M.
\end{align*}
Further, we have
\begin{align*}
N:=\Big\{\big(\frac{r_{\sigma(1)}}{\sum_{j=1}^nr_{\sigma(j)}},\ldots, \frac{r_{\sigma(n)}}{\sum_{j=1}^nr_{\sigma(j)}}\big):\sigma\in \mathfrak{P}_n\Big\}\subset M
\end{align*}
and $N$ has cardinality $|N|=n!$. Since $x\mapsto \frac{x}{x+c}$ is strictly increasing in $x>0$ for arbitrary but fixed $c>0$, we see that
\begin{align*}
\frac{r_{n+1}}{(\sum_{j=1}^{n-1}r_j)+r_{n+1}}>\frac{r_{n}}{(\sum_{j=1}^{n-1}r_j)+r_{n}},
\end{align*}
 which shows that 
\begin{align*}
\Big(\frac{r_{1}}{(\sum_{j=1}^{n-1}r_j)+r_{n+1}},\ldots,\frac{r_{n-1}}{(\sum_{j=1}^{n-1}r_j)+r_{n+1}},\frac{r_{n+1}}{(\sum_{j=1}^{n-1}r_j)+r_{n+1}}\Big)\in M\backslash N,
\end{align*}
and, consequently, $|\mathcal{S}|\geq |M|>|N|=n!$, contradicting the assumption $\mathcal{S}=[\vec{q}]$.
\end{proof}

If we write $\mathfrak{L}^X_{n,d}$ for the subset of those stable tail dependence functions in $\mathfrak{L}^X_n$ which arise as $n$-margin of some element in $\mathfrak{L}^X_d$, Lemma \ref{Theorem_impossible} shows for arbitrary $k \in \IN$ that 
\begin{gather*}
\bigcap_{d \geq n} \mathfrak{L}^X_{n,d} \subsetneq \mathfrak{L}^X_{n,n+k} \subsetneq  \mathfrak{L}^X_{n}.
\end{gather*}
Finally, recall that for $n=2$ the set $\bigcap_{d \geq 2} \mathfrak{L}^X_{2,d}$ is essentially equal to $\mathfrak{A}^{\ast}$, since $A \in \mathfrak{A}^{\ast}$ is defined from $\ell \in \bigcap_{d \geq 2} \mathfrak{L}^X_{2,d}$ via $A(x)=\ell(x,1-x)$, $x \in [0,1]$.

\begin{figure}[!ht]
\includegraphics[width=7cm]{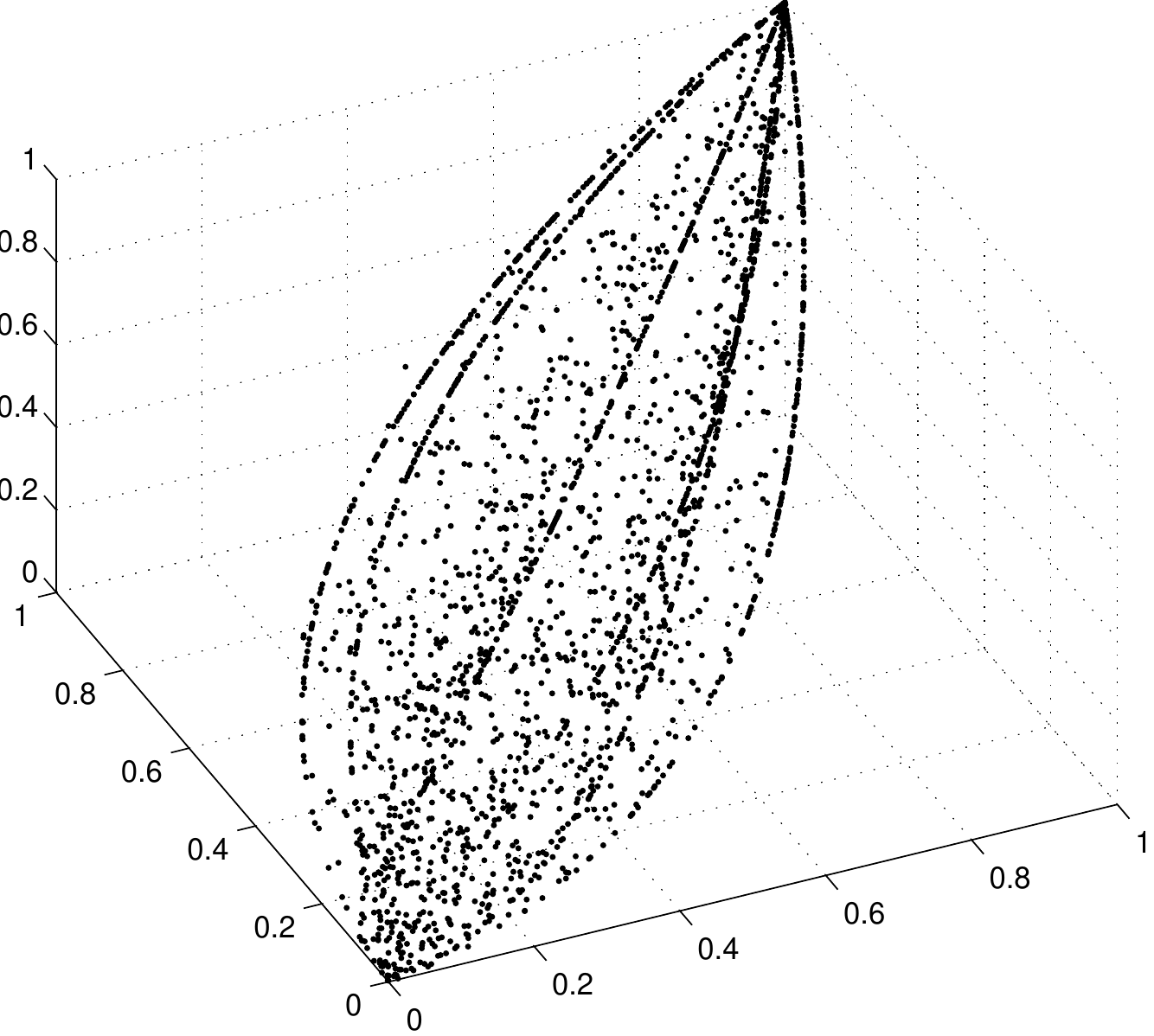}
\includegraphics[width=7cm]{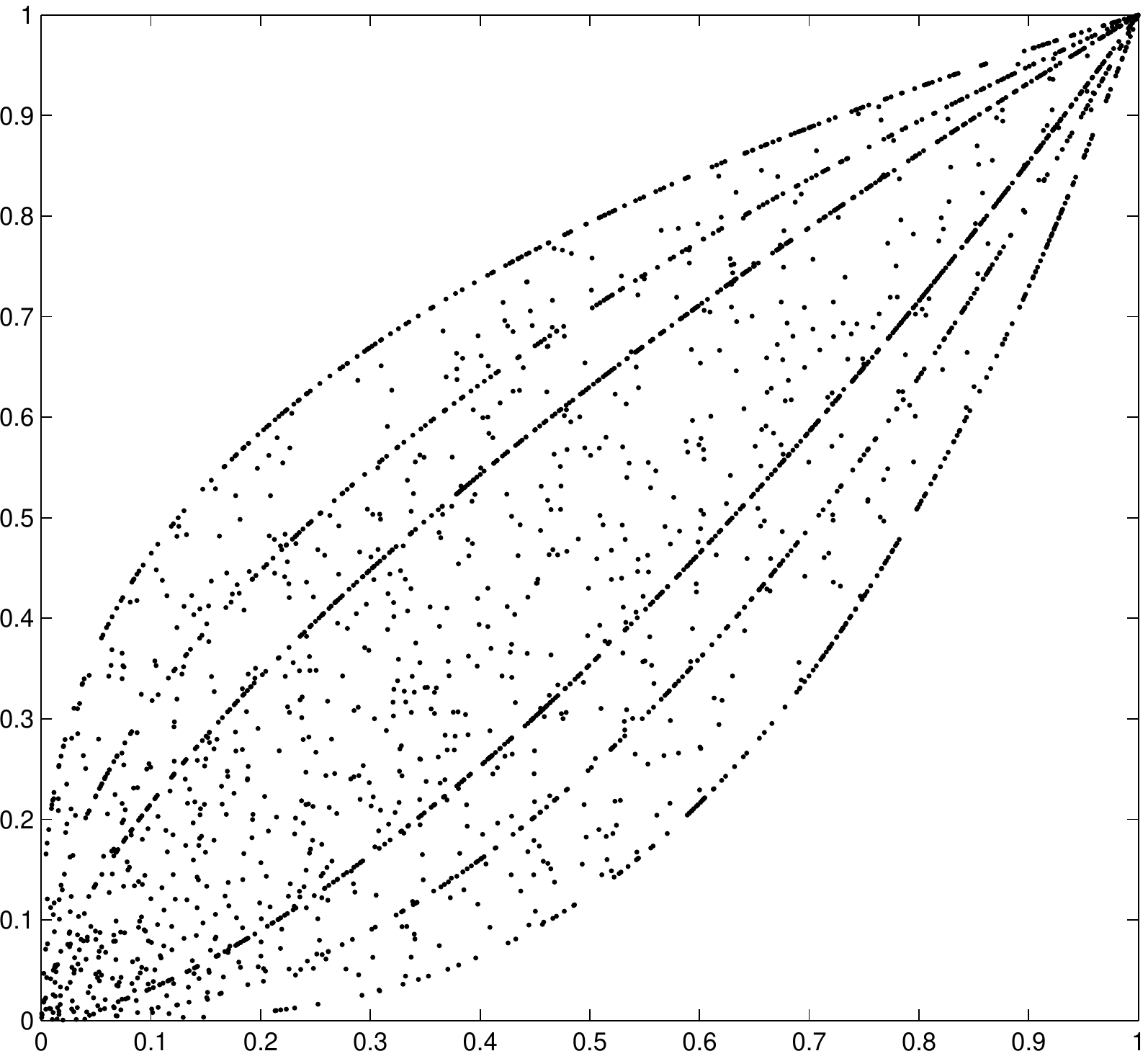}
\caption{Left: Scatter plot of $2,500$ simulated points $(U_1,U_2,U_3)$ from the extreme-value copula associated with $\ell^{(3)}_{[(1/6,1/3,1/2)]}$. Right: Scatter plot of $2,500$ simulated points of the first two coordinates $(U_1,U_2)$ of the three-dimensional points in the left plot. Notice the singular component which is concentrated on six one-dimensional paths from $(0,0)$ to $(1,1)$.}
\label{figextr}
\end{figure}

Figure \ref{figextr} depicts a scatter plot for the extreme-value copula associated with the extremal element $\ell^{(3)}_{[(1/6,1/3,1/2)]} \in \mathfrak{L}^X_3$, as well as a scatter plot of a bivariate margin thereof. We know from \cite{maischerer11} that extremal elements of $\mathfrak{L}^X_2$ induce extreme-value copulas with singular component which is concentrated on (at most) two one-dimensional paths from $(0,0)$ to $(1,1)$. In contrast, the scatter plot on the right-hand side apparently exhibits singular support on six such paths. This confirms our finding that bivariate exchangeable extreme-value copulas which are extendible to a three-variate exchangeable extreme-value copula cannot have Pickands measure concentrated on at most two atoms. 

\section{Infinite extendibility of bivariate extreme-value copulas}\label{Section_3}
Whereas Lemma~\ref{Theorem_impossible} resolves our `Moby Dick' in the negative, this raises the next natural question: Given some $A \in \mathfrak{A}^X$, can we find a useful analytical criterion to solve the membership testing problem $A\in \mathfrak{A}^{\ast}$? In this section, we work towards a solution by deriving a useful necessary condition.
\par
From (\ref{condiidcopula}) we know that $\mathfrak{A}^{\ast}$ equals the convex hull of $\{A_{0,1}\} \cup \{A_F\,:\,F \in \mathfrak{F}_1\}$, where
\begin{gather*}
\mathfrak{F}_1:=\{F:[0,\infty)\rightarrow [0,1]\,:\,F \mbox{ is d.f.\ of a r.v.\ }X \geq 0 \mbox{ with }\IE[X]=1\} 
\end{gather*}
and 
\begin{gather*}
A_F(x) := \int_0^{\infty}1-F\Big( \frac{s}{x}\Big)\,F\Big( \frac{s}{1-x}\Big)\,\mathrm{d}s.
\end{gather*}
The constant Pickands dependence function $A_{0,1}$ represents independence and lies in the closure\footnote{With respect to the topology of pointwise convergence.} of $\{A_F\,:\,F \in \mathfrak{F}_1\}$, which can be seen with the help of Example \ref{ex_CA} below in the case $\theta \rightarrow 0$. Thus, the set $\mathfrak{A}^{\ast}$ is the closed convex hull of $\{A_F\,:\,F \in \mathfrak{F}_1\}$. Consequently, it is of fundamental importance to understand properties of $A_F$ for $F \in \mathfrak{F}_1$. 
\par
Denoting by $Q=Q_F$ the random variable associated with $A_F$ via (\ref{Pick_biv}), we know from \cite{mai18b} that  
\begin{gather*}
Q_F \stackrel{d}{=} \frac{M_F}{X_F+Y_F},\quad \,M_F = \begin{cases}
X_F & \mbox{, if }{B} = 1\\
Y_F & \mbox{, if }{B}=0\\
\end{cases},
\end{gather*}
where the three random variables $B \sim \text{Bernoulli}(1/2)$, $X_F \sim F$, and $Y_F \sim x\,\mathrm{d}F(x)$ are independent. One readily verifies for $0<q < 1/2$ that 
\begin{gather}
\IP(Q_F \leq q) = \frac{1}{2}\,\IP\Bigg( \frac{X_F}{Y_F} \in \Big(\frac{q}{1-q},\frac{1-q}{q} \Big)^c\Bigg), \label{special}
\end{gather}
noticing that $Y_F>0$ almost surely, even though $X_F$ might be zero (so the denominator is well-defined). Notice in particular that (\ref{special}) implies that 
\begin{gather*}
\IP\Big(Q_F = \frac{1}{2}\Big) = 1-2\,\IP\Big( Q_F < \frac{1}{2}\Big)=1-\IP\Big( \frac{X_F}{Y_F} \in \{1\}^c\Big)=\IP\Big( \frac{X_F}{Y_F} =1\Big).
\end{gather*}
Furthermore, the random variable $1/Y_F$ is easily seen to have unit expectation, which implies that $\IE[X_F/Y_F]=1$. 

\begin{example}[An extendible example: The Cuadras--Aug\'e copula]\label{ex_CA}
For $\theta \in (0,1]$ we consider $A(x)=(1-\theta)+\theta\,(1-x)$, $x \leq 1/2$, i.e.\ a convex mixture of independence ($A(x) \equiv 1$) and co-monotonicity ($A(x)=1-x$, $x\leq 1/2$). We have that $A \in \mathfrak{A}^{\ast}$, since both independence and co-monotonicity are conditionally iid. Alternatively, we have that $A=A_F$ for $F(x) = 1-\theta+\theta\,\I_{\{x \geq 1/\theta\}} \in \mathfrak{F}_1$, as will be explained briefly. Using the notation $X_F,Y_F$ from above we observe that $\IP(X_F=0)=1-\theta=1-\IP(X_F=1/\theta)$ and $Y_F \equiv 1$. Consequently, we observe $\IP(X_F/Y_F=0)=\IP(X_F=0)=1-\theta$ and $\IP(X_F/Y_F=1/\theta)=\IP(X_F=1/\theta)=\theta$. In particular, Formula (\ref{special}) implies that
\begin{gather*}
\IP(Q_F \leq q) = \frac{1-\theta}{2},\quad 0<q<\frac{1}{2}.
\end{gather*}
This implies that 
\begin{gather*}
Q_F = \begin{cases}
0 & \mbox{, with probability }\frac{1-\theta}{2}\\
\frac{1}{2} & \mbox{, with probability }\theta\\
1 & \mbox{, with probability }\frac{1-\theta}{2}\\
\end{cases}.
\end{gather*}
Indeed, it is verified that
\begin{gather*}
2\,\IE[\max\{x\,Q_F,(1-x)\,(1-Q_F)\}] = 1-\theta+\theta\,(1-x)=A(x),\quad x \leq \frac{1}{2},
\end{gather*}
as claimed.
\end{example}

The following result gives a non-trivial necessary condition for a symmetric random variable $Q$ on $[0,1]$ to define via (\ref{Pick_biv}) a Pickands dependence function in $\mathfrak{A}^{\ast}$.
\begin{lemma}[Necessary condition for extendibility]\label{lemma_nec}
Let $Q$ be a random variable on $[0,1]$ with $Q \stackrel{d}{=} 1-Q$, and denote by $A(x):=2\,\IE[\max\{x\,Q,(1-x)\,(1-Q)\}]$ its associated Pickands dependence function in $\mathfrak{A}^X$. We assume furthermore that $A \in \mathfrak{A}^{\ast}$.
\begin{itemize}
\item[(a)] \textbf{Discrete case:} If the support of $Q$ consists of finitely many values,
\begin{gather*}
\IP(Q = q) \leq \frac{\IP(Q=1/2)}{2\,\sqrt{q\,(1-q)}},\quad q \in [0,1].
\end{gather*}
\item[(b)] \textbf{Absolutely continuous case:} If $Q$ is absolutely continuous with continuous density $f_Q$, 
\begin{gather*}
f_Q(q) \leq \frac{f_Q(1/2)}{8\,\sqrt{q^3\,(1-q)^3}},\quad q \in (0,1).
\end{gather*}
\end{itemize}
\end{lemma}

\begin{remark}[Intuition of the necessary conditions]
Intuitively, belonging to $\mathfrak{A}^{\ast}$ requires the probability law of $Q$ to allocate mass near the three values $0,1/2,1$. 
\begin{itemize}
\item The (discrete or continuous) density is bounded from above by a symmetric function, which takes its minimum at $1/2$ and increases monotonically to infinity for $q \rightarrow 0$ and $q \rightarrow 1$. In particular, the less mass one has at $q=1/2$, the more restrictive becomes the bound.
\item Conversely, Lemma \ref{lemma_nec} gives a lower bound on the (discrete or continuous) density value at $q=1/2$, to wit
\begin{align*}
\IP(Q=1/2)& \geq 2\,\sup_{q \in [0,1]}\big\{\IP(Q=q)\,\sqrt{q\,(1-q)}\big\} \geq 0,\\
f_{Q_F}(1/2) & \geq 8\,\sup_{0<q<1}\big\{f_{Q_F}(q)\,\sqrt{q^3\,(1-q)^3}\big\}>0.
\end{align*}
In intuitive terms, this lower bound only becomes trivial (zero or very small), if all probability mass is pushed near the boundaries zero and one, meaning that $A \approx A_{0,1}$. If one wishes to put mass away from these boundaries, necessarily one has to put mass at $1/2$.
\end{itemize}
\end{remark}
\begin{proof}[of Lemma \ref{lemma_nec}]
Since $\mathfrak{A}^{\ast}$ is the closed convex hull of the $A_F$, $F \in \mathfrak{F}_1$, it is sufficient to prove both statements in the special case $Q=Q_F$ for some arbitrary $F \in \mathfrak{F}_1$, as the claimed inequalities remain valid under convex mixtures. We first consider the discrete case. The support of  $Q_F$ is finite if and only if $X_F \sim F$ takes values $0 \leq x_1<x_2<\ldots<x_n$ with respective probabilities $p_1,\ldots,p_n$. The condition $F \in \mathfrak{F}_1$ implies $\sum_{i=1}^{n}p_i=1=\sum_{i=1}^{n}p_i\,x_i$. Furthermore, it is not difficult to verify from Equation~(\ref{special}) that
\begin{gather*}
\IP\Big(Q_F =\frac{x_i}{x_i+x_j}\Big) = \frac{x_i+x_j}{2}\,p_i\,p_j,\quad 1 \leq i,j \leq n. 
\end{gather*}
By symmetry it suffices to study the probability mass of $Q_F$ on $[0,1/2]$. Aggregating the cases $i=j$ implies that $Q_F$ takes the value $1/2$ with non-zero probability $\sum_{i=1}^{n}p_i^2\,x_i$. Furthermore, the potential values of $Q_F$ in $[0,1/2)$ are $x_i/(x_i+x_j)$ with $i<j$, but some of these values might coincide. Fixing one particular $0 \leq q<1/2$, we denote by $(i,j) \leadsto q$ the set of all index pairs $(i,j)$ satisfying $x_i/(x_i+x_j)=q$ (notice that $i<j$ since $q<1/2$), and observe
\begin{align*}
\IP(Q_F =q) = \sum_{(i,j) \leadsto q}\frac{x_i+x_j}{2}\,p_i\,p_j  = \frac{\sum_{(i,j) \leadsto q}\sqrt{x_i}\,p_i\,\sqrt{x_j}\,p_j}{2\,\sqrt{q\,(1-q)}},
\end{align*}
where the last equality uses the fact that $x_i=(x_i+x_j)\,q$ if and only if $(x_i+x_j)^2=x_i\,x_j/(q\,(1-q))$. Using the Cauchy--Schwarz inequality, the sum in the last expression can be estimated from above by $\sum_{i=1}^{n}p_i^2\,x_i=\IP(Q_F=1/2)$, so that we obtain
\begin{gather*}
\IP(Q_F = q) \leq \frac{\IP(Q_F=1/2)}{2\,\sqrt{q\,(1-q)}}, \quad q \in [0,1/2).
\end{gather*}
By symmetry, this inequality clearly also holds for $q \in (1/2,1]$ and for $q=1/2$ it is an equality. 
\par
Next, we consider the continuous case. Clearly, $Q_F$ has a continuous density if and only if $F \in \mathfrak{F}_1$ has a continuous density $f$. The random variable $X_F/Y_F$ has density
\begin{gather}
f_{\frac{X_F}{Y_F}}(z) = \int_0^{\infty}x^2\,f(x\,z)\,f(x)\,\mathrm{d}x,\quad z \in (0,\infty),
\label{denstrans}
\end{gather}
and it follows from Equation~(\ref{special}) that $Q_F$ has (symmetric) density
\begin{gather*}
f_{Q_F}(q)=\frac{1}{2}\,\Big\{ f_{\frac{X_F}{Y_F}}\Big( \frac{q}{1-q}\Big)\,\frac{1}{(1-q)^2}+ f_{\frac{X_F}{Y_F}}\Big( \frac{1-q}{q}\Big)\,\frac{1}{q^2}\Big\},\quad q \in (0,1).
\end{gather*}
It is not difficult to observe from Equation~(\ref{denstrans}) for $q \in (0,1/2)$ that 
\begin{gather*}
\frac{f_{\frac{X_F}{Y_F}}\big(\frac{q}{1-q}\big)}{(1-q)^3} = \frac{f_{\frac{X_F}{Y_F}}\big(\frac{1-q}{q}\big)}{q^3},
\end{gather*}
which implies that
\begin{gather*}
f_{Q_F}(q) = \frac{1}{2\,q^3}\,f_{\frac{X_F}{Y_F}}\Big( \frac{1-q}{q}\Big)=\frac{1}{2}\,\int_0^{\infty}x^2\,f(q\,x)\,f((1-q)\,x)\,\mathrm{d}x,\quad q \in (0,1).
\end{gather*}
Applying the Cauchy--Schwarz-inequality and substitution, we observe a non-trivial restriction on $f_{Q_F}$, to wit
\begin{align*}
f_{Q_F}(q) & \leq \frac{1}{2}\,\sqrt{\int_0^{\infty}x^2\,\big(f(q\,x)\big)^2\,\mathrm{d}x\,\int_0^{\infty}x^2\,\big(f((1-q)\,x)\big)^2\,\mathrm{d}x}\\
& = \frac{f_{Q_F}(1/2)}{8\,\sqrt{q^3\,(1-q)^3}},\quad 0<q<1,
\end{align*}
as claimed.
\end{proof}

\begin{example}[Alternative proof for $A_{q,1-q} \notin \mathfrak{A}^{\ast}$]
Fix $0<q<1/2$ and consider the Pickands dependence function $A_{q,1-q}$ from Example~\ref{ex_BC2}. Recall that Lemma \ref{Theorem_impossible} in the case $n=2$ shows that $A_{q,1-q} \notin \mathfrak{A}^{\ast}$, since $A_{q,1-q}(x)=\ell_{[(q,1-q)]}(x,1-x)$. An alternative proof can be retrieved immediately from the necessary condition in Lemma \ref{lemma_nec}(a). Recall from Example \ref{ex_BC2} that the random vector $Q=Q_{q,1-q}$ associated with $A_{q,1-q}$ has finite support, to wit
\begin{gather*}
\IP(Q= q) = \frac{\frac{1}{2}-q}{1-2\,q} =1-\IP(Q=1-q).
\end{gather*}
Apparently, $\IP(Q=1/2)=0$ and the necessary condition of Lemma \ref{lemma_nec}(a) is violated.
\end{example}

\section{Conclusion}
We have shown that $d$-variate, symmetric stable tail dependence functions form a simplex. Furthermore, it was shown that there are $d$-variate, symmetric stable tail dependence functions which do not arise as $d$-margins of some higher-dimensional symmetric stable tail dependence function. In particular, not every exchangeable extreme-value copula is infinitely extendible, not even in the bivariate case. Moreover, we have provided a useful and intuitive necessary criterion for a bivariate extreme-value copula to be infinitely extendible.


\begin{thebibliography}{}
\bibitem[Beran, Mainik (2014)]{beran14} J.\ Beran, G.\ Mainik, On estimating extremal dependence structures by parametric spectral measures, {Statistical Methodology} {21} (2014) pp.\ 1--22. 
\bibitem[De Finetti (1937)]{definetti37} B.\ De Finetti, La pr\'evision: ses lois logiques, ses sources subjectives, { Annales de l'Institut Henri Poincar\'e} { 7} (1937) pp.\ 1--68.
\bibitem[De Haan (1984)]{dehaan84} L.\ De Haan, A spectral representation for max-stable processes, {Annals of probability} {12:4} (1984) pp.\ 1194--1204.
\bibitem[Durante, Sempi (2015)]{durante2015principles} F.\ Durante, C.\ Sempi, {Principles of Copula Theory}, {CRC Press} (2015).
\bibitem[Gnedin (1995)]{gnedin95} A.V.\ Gnedin, On a class of exchangeable sequences, Statistics and Probability Letters 25 (1995) pp.\ 351--355.
\bibitem[Joe (1997)]{joe97} H.\ Joe, {Multivariate Models and Dependence Concepts}, {Chapman and Hall/CRC, London} (1997).
\bibitem[Joe (2014)]{joe2014dependence} H.\ Joe, {Dependence Modeling with Copulas}, {CRC Press, London} (2014).
\bibitem[Konstantopoulos, Yuan (2019)]{konstantopoulos19} T.\ Konstantopoulos, L.\ Yuan, On the extendibility of finitely exchangeable probability measures, {Transactions of the American Mathematical Society} 371 (2019) pp.\ 7067--7092.
\bibitem[Mai (2019)]{mai18b} J.-F.\ Mai, Canonical spectral representation for exchangeable max-stable sequences, {Extremes, in press} (2019).
\bibitem[Mai, Scherer (2011)]{maischerer11} J.-F.\ Mai, M.\ Scherer, Bivariate extreme-value copulas with discrete Pickands dependence measure, {Extremes} {14} (2011) pp.\ 311--324.
\bibitem[Mai et al.\ (2013)]{mai13} J.-F.\ Mai, M.\ Scherer, N.\ Shenkman, Multivariate geometric laws, (logarithmically) monotone sequences, and infinitely divisible laws, {Journal of Multivariate Analysis} {115} (2013) pp.\ 457--480.
\bibitem[Mai, Scherer (2013)]{maischerer13brazil} J.-F.\ Mai, M.\ Scherer, Extendibility of Marshall-Olkin distributions and inverse Pascal triangles, {Brazilian Journal of Probability and Statistics} {27} (2013) pp.\ 310--321.
\bibitem[Mai, Scherer (2014)]{maischerer13} J.-F.\ Mai, M.\ Scherer, Characterization of extendible distributions with exponential minima via processes that are infinitely divisible with respect to time, {Extremes} {17} (2014) pp.\ 77--95.
\bibitem[McNeil, Ne\v{s}lehov\'{a} (2009)]{neilnev09} A.J.\ McNeil, J.\ Ne\v{s}lehov\'{a}, Multivariate Archimedean copulas, $d$-monotone functions and $l_{1}$-norm symmetric distributions, {Annals of Statistics} {37}(5B) (2009) pp.\ 3059--3097.
\bibitem[Pickands (1981)]{Pickands1981} J.\ Pickands, Multivariate extreme-value distributions. In: Proceedings of the 43rd Session of the International Statistical Institute, volume 2 pp.\ 859--878 (1981).
\bibitem[Rachev, R\"uschendorf (1991)]{rachev91} S.T.\ Rachev, L.\ R\"uschendorf, Approximate independence of distributions on spheres and their stability properties, { Annals of Probability} {19} (1991) pp.\ 1311--1337.
\bibitem[Resnick (1987)]{resnick87} S.I.\ Resnick, Extreme-values, regular variation and point processes, Springer, New York, 1987.
\bibitem[Ressel (2013)]{ressel13} P.\ Ressel, Homogeneous distributions and a spectral representation of classical mean values and stable tail dependence functions, Journal of Multivariate Analysis 117 (2013) pp.\ 246--256.
\bibitem[Shaked (1977)]{shaked77} M.\ Shaked, A concept of positive dependence for exchangeable random variables, { Annals of Statistics} { 5} (1977) pp.\ 505--515.
\bibitem[Spizzichino (1982)]{spizzichino82} F.\ Spizzichino, Extendibility of symmetric probability distributions and related bounds, {in Exchangeability in Probability and Statistics, edited by G.\ Koch and F.\ Spizzichino, North-Holland Publishing Company} (1982) pp.\ 313--320.
\bibitem[Trutschnig et al.\ (2016)]{trutschnig16} W.\ Trutschnig, M.\ Schreyer, J.\ Fern\'{a}ndez S\'{a}nchez, Mass distributions of two-dimensional extreme-value copulas and related results, {Extremes} {19} (2016) pp.\ 405--427. 
\end{thebibliography}
\end{document}